\documentclass[11pt,oneside,reqno,french,ngerman,russian,english]{amsart}
\usepackage[T2A,T1]{fontenc}
\usepackage[koi8-r,latin9]{inputenc}
\usepackage{color}
\usepackage{babel}
\makeatletter
\addto\extrasfrench{%
   \providecommand{\fg}{\ifdim\lastskip>\z@\unskip\fi~\frqq}%
}

\makeatother
\usepackage{units}
\usepackage{mathtools}
\usepackage{amstext}
\usepackage{amsthm}
\usepackage{amssymb}
\usepackage{setspace}
\usepackage[unicode=true,pdfusetitle,
 bookmarks=true,bookmarksnumbered=false,bookmarksopen=false,
 breaklinks=true,pdfborder={0 0 0},pdfborderstyle={},backref=false,colorlinks=true]
 {hyperref}

\makeatletter

\DeclareRobustCommand{\cyrtext}{%
  \fontencoding{T2A}\selectfont\def\encodingdefault{T2A}}
\DeclareRobustCommand{\textcyr}[1]{\leavevmode{\cyrtext #1}}

\theoremstyle{plain}
\newtheorem{thm}{\protect\theoremname}
\theoremstyle{plain}
\newtheorem{prop}[thm]{\protect\propositionname}
\theoremstyle{plain}
\newtheorem{lem}[thm]{\protect\lemmaname}
\theoremstyle{remark}
\newtheorem*{rem*}{\protect\remarkname}
\theoremstyle{remark}
\newtheorem*{claim*}{\protect\claimname}

\usepackage{calrsfs}
\usepackage{graphicx}
\usepackage{color}
\usepackage{perpage}
\usepackage{upquote}
\usepackage{bbm}
\usepackage[shortlabels]{enumitem}
\usepackage{stmaryrd} 
\usepackage{breakurl}

\MakePerPage{footnote}
\gdef\SetFigFontNFSS#1#2#3#4#5{} 
\gdef\SetFigFont#1#2#3#4#5{} 
\def\clap#1{\hbox to 0pt{\hss#1\hss}}
\def\colonqq{\mathrel{\mathop:}=}

\DeclareMathOperator{\supp}{supp}

\DeclareMathOperator{\Mink}{Mink}
\DeclareMathOperator{\capa}{cap}

\definecolor{myblue}{rgb}{0.09,0.32,0.44} 
\hypersetup{pdfborder={0 0 0},pdfborderstyle={},colorlinks=true,linkcolor=myblue,citecolor=myblue,urlcolor=blue}

\theoremstyle{remark}
\newtheorem*{qst*}{Question}
\newtheorem*{rmrks*}{Remarks}
\theoremstyle{plain}
\newtheorem*{assumption}{Assumption}

\newlength{\tempindent} 
\newcommand{\lazyenum}{
\setlength{\tempindent}{\parindent} 
\begin{enumerate}[leftmargin=0cm,itemindent=0.7cm,labelwidth=\itemindent,labelsep=0cm,align=left,label=\arabic*)]
\setlength{\parskip}{\smallskipamount}
\setlength{\parindent}{\tempindent}
}

\makeatletter
\renewcommand{\andify}{%
  \nxandlist{\unskip, }{\unskip{} \@@and~}{\unskip{} \@@and~}}
\def\author@andify{%
  \nxandlist {\unskip ,\penalty-1 \space\ignorespaces}%
    {\unskip {} \@@and~}%
    {\unskip \penalty-2 \space \@@and~}%
}
\let\@wraptoccontribs\wraptoccontribs
\makeatother

\ifpdf
\def\afs#1#2{\href{#1}{\nolinkurl{#2}}}
\else
\def\afs#1#2{\burlalt{#1}{#2}}
\fi

\makeatother

\addto\captionsenglish{\renewcommand{\claimname}{Claim}}
\addto\captionsenglish{\renewcommand{\lemmaname}{Lemma}}
\addto\captionsenglish{\renewcommand{\propositionname}{Proposition}}
\addto\captionsenglish{\renewcommand{\remarkname}{Remark}}
\addto\captionsenglish{\renewcommand{\theoremname}{Theorem}}
\addto\captionsfrench{\renewcommand{\claimname}{Affirmation}}
\addto\captionsfrench{\renewcommand{\lemmaname}{Lemme}}
\addto\captionsfrench{\renewcommand{\propositionname}{Proposition}}
\addto\captionsfrench{\renewcommand{\remarkname}{Remarque}}
\addto\captionsfrench{\renewcommand{\theoremname}{Théorème}}
\addto\captionsngerman{\renewcommand{\claimname}{Behauptung}}
\addto\captionsngerman{\renewcommand{\lemmaname}{Lemma}}
\addto\captionsngerman{\renewcommand{\propositionname}{Satz}}
\addto\captionsngerman{\renewcommand{\remarkname}{Bemerkung}}
\addto\captionsngerman{\renewcommand{\theoremname}{Theorem}}
\addto\captionsrussian{\renewcommand{\claimname}{\inputencoding{koi8-r}õÔ×ÅÒÖÄÅÎÉÅ}}
\addto\captionsrussian{\renewcommand{\lemmaname}{\inputencoding{koi8-r}ìÅÍÍÁ}}
\addto\captionsrussian{\renewcommand{\propositionname}{\inputencoding{koi8-r}ðÒÅÄÌÏÖÅÎÉÅ}}
\addto\captionsrussian{\renewcommand{\remarkname}{\inputencoding{koi8-r}úÁÍÅÞÁÎÉÅ}}
\addto\captionsrussian{\renewcommand{\theoremname}{\inputencoding{koi8-r}ôÅÏÒÅÍÁ}}
\providecommand{\claimname}{Claim}
\providecommand{\lemmaname}{Lemma}
\providecommand{\propositionname}{Proposition}
\providecommand{\remarkname}{Remark}
\providecommand{\theoremname}{Theorem}

\begin{document}
\title[uniqueness along subsequences]{Cantor uniqueness and multiplicity along subsequences}
\author{Gady Kozma and Alexander Olevski\u\i}
\address{GK: Weizmann institute of Science, Rehovot, Israel.}
\email{gady.kozma@weizmann.ac.il}
\address{AO: Tel Aviv University, Tel Aviv, Israel}
\email{olevskii@post.tau.ac.il}
\begin{abstract}
We construct a sequence $c_{l}\to0$ such that the trigonometric series
$\sum c_{l}e^{ilx}$ converges to zero everywhere on a subsequence
$n_{k}$. We show that any such series must satisfy that the $n_{k}$
are very sparse, and that the support of the related distribution
is quite large.
\end{abstract}

\maketitle

\section{Introduction}

In 1870, Georg Cantor proved his famous uniqueness theorem for trigonometric
series: if a series $\sum c_{l}e^{ilx}$ converges to zero for every
$x\in[0,2\pi]$, then the $c_{l}$ are all zero \cite{C1870}. The
proof used important ideas from Riemann's \emph{Habilitationsschrift},
namely, that of taking the formal double integral $F(x)=\sum\frac{1}{l^{2}}c_{l}e^{ilx}$
and examining the second Schwarz derivative of $F$. Cantor's proof
is now classic and may be found in many books, e.g.\ \cite[\S IX]{Z}
or \cite[\S XIV]{B64}. A fascinating historical survey of these early
steps in uniqueness theory, including why Riemann defined $F$ in
the first place, may be found in \cite{C93}. (briefly, Riemann was
writing \emph{necessary} conditions for a function to be represented
by a trigonometric series in terms of its double integral).

Cantor's result may be extended in many directions, and probably the
most famous one was the direction taken by Cantor himself, that of
trying to see if the theorem still holds if the series is allowed
not to converge at a certain set, which led Cantor to develop set
theory, and led others to the beautiful theory of sets of uniqueness,
see \cite{KL87}. But in this paper we are interested in a different
kind of extension: does the theorem hold when the series $\sum c_{l}e^{ilx}$
is required to converge only on a subsequence?

This problem was first tackled in 1950, when Kozlov constructed a
nontrivial sequence $c_{l}$ and a second sequence $n_{k}$ such that
\begin{equation}
\lim_{k\to\infty}\sum_{l=-n_{k}}^{n_{k}}c_{l}e^{ilx}=0\qquad\forall x\in[0,2\pi].\label{eq:1}
\end{equation}
See \cite{K50} or \cite[\S XV.6]{B64}. A feature of Kozlov's construction
that was immediately apparent is that the coefficients $c_{l}$ are
(at least for some $l$), very large. Therefore it was natural to
ask if it is possible to have (\ref{eq:1}) together with $c_{l}\to0$.
The problem was first mentioned in the survey of Talalyan \cite{T60}
\textemdash{} this is problem 13 in \S10 (note that there is a mistake
in the English translation), and then repeated in \cite{AT64} where
the authors note, on page 1406, that the problem is ``very hard''.
In the same year, the survey of Ulyanov \cite[page 20 of the English version]{U64}
mentions the problem and conjectures that in fact, no such series
exists. Skvortsov constructed a counterexample for the Walsh system
\cite{S75}, but not for the Fourier system.

\subsection{Results}

In this paper we answer this question in the positive. Here is the
precise statement:
\begin{thm}
\label{thm:example}There exist coefficients $c_{l}\to0$, not all
zero, and $n_{k}\to\infty$ such that (\ref{eq:1}) holds.
\end{thm}

The existence of such an example raises many new questions about the
nature of the $c_{l}$, of the distribution $\sum c_{l}e^{ilx}$,
and of the numbers $n_{k}$. We have two results which show some restrictions
on these objects. The first states, roughly, that the $n_{k}$ must
increase at least doubly exponentially:
\begin{thm}
\label{thm:doubly exponentially}Let $c_{l}\to0$ and let $n_{k}$
be such that (\ref{eq:1}) holds. Assume further that $n_{k+1}=n_{k}^{1+o(1)}$.
Then $c_{l}\equiv0$.
\end{thm}

Our second result is a lower bound on the dimension of the support
of the distribution $\sum c_{l}e^{ilx}$. It is stated in terms of
the upper Minkowski dimension (see, e.g., \cite{F14} where it is
called the box counting dimension) which we denote by $\dim_{\textrm{Mink}}$.
\begin{thm}
\label{thm:dim}Let $c_{l}\to0$ and let $n_{k}$ be such that (\ref{eq:1})
holds. Let $K$ be the support of the distribution $\sum c_{l}e^{ilx}$,
and assume that 
\[
\dim_{\Mink}(K)<\frac{1}{2}(\sqrt{17}-3)\approx0.561.
\]
Then $c_{l}\equiv0$.
\end{thm}

\subsection{\label{subsec:Comments}Comments and questions}

An immediate question is the sharpness of the double exponential bound
of theorem \ref{thm:doubly exponentially}. The proof of theorem \ref{thm:example}
which we will present is not quantitative, but it can be quantified
with only a modicum of effort, giving:
\begin{quote}
\emph{There exists $c_{l}\to0$ and $n_{k}=\exp(\exp(O(k))$ such
that (\ref{eq:1}) holds.}
\end{quote}
(this quantitative version, and all other claims in this section,
\S \ref{subsec:Comments}, will not be proved in this paper). Thus
in this setting the main problem remaining is the constant in the
exponent. The reader might find it useful to think about the question
as follows: suppose $n_{k+1}=n_{k}^{\lambda}$. For which value of
$\lambda$ is it possible to construct a counterexample with this
$n_{k}$?

But an even more interesting question is: what happens when the condition
$c_{l}\to0$ is removed from theorem \ref{thm:doubly exponentially}?
The answer is no longer doubly exponential, in fact Nina Bary \cite{B60}
showed that one can take $n_{k}$ growing only slightly faster than
exponentially, and conjectured that this rate of growth is optimal.
Our techniques allow only modest progress towards Bary's conjecture:
we can show that if $n_{k+1}-n_{k}=o(\log k)$ then no such example
may exist.

A variation of the problem where our upper and lower bounds match
more closely is the following: suppose that we require $c_{l}=0$
for all $l<0$ (often this is called an ``analytic'' version of
the problem, because there is a naturally associated analytic function
in the disk, $\sum c_{l}z^{l}$). In this case, the following can
be proved. On the one hand, one can extend Bary's construction and
find an example of a $c_{l}$ and $n_{k}$ both growing slightly faster
than exponential such that 
\[
\lim_{k\to\infty}\sum_{l=0}^{n_{k}}c_{l}e^{ilx}=0\qquad\forall x.
\]
On the other hand, it is not possible to have such an example if either
$n_{k+1}-n_{k}\le n_{k}/\log^{2}n_{k}$ or $|c_{l}|\le\exp(Cl/\log^{2}l)$.
This holds for any inverse of a non-quasianalytic sequence.

In a different direction, the condition $c_{l}\to0$ can be improved:
it is possible to require, in theorem \ref{thm:example}, that the
coefficients $c_{l}$ be inside $\ell^{2+\epsilon}$, for any $\epsilon>0$.
This, too, will not be shown in this paper, but the proof is a simple
variation on the proof of theorem \ref{thm:example} below.

Another interesting question is the sharpness of the dimension bound
in theorem \ref{thm:dim}. In our example the dimension of the support
is $1$ (even for the Hausdorff dimension, which is smaller than the
upper Minkowski dimension). It would be very interesting to construct
an example with dimension strictly smaller than $1$. In the opposite
direction, let us remark that in our example the support of the distribution
$\sum c_{l}e^{ilx}$ has measure zero, but it is not difficult to
modify the example so that the support would have positive measure.
We find this interesting because this distribution is so inherently
singular. The support must always be nowhere dense, see lemma \ref{lem:K}
below.

\subsection{Measures}

It is interesting to note that the proofs of theorem \ref{thm:doubly exponentially}
and \ref{thm:dim} do not use the Riemann function in any way. In
fact, the only element of classic uniqueness theory that appears in
the proof is the localisation principle, in the form of Rajchman (see
\S \ref{subsec:Rajchman}). Thus the proof of theorem \ref{thm:doubly exponentially}
is also a new proof of Cantor's classic result. In the 150 years that
passed since its original publication, the only other attempt we are
aware of is \cite{A89}, which gives a proof of Cantor's theorem using
one formal integration rather than two. To give the reader a taste
of the ideas in the proofs of theorems \ref{thm:doubly exponentially}
and \ref{thm:dim}, let us apply the same basic scheme to prove a
simpler result: that no such construction is possible with $c_{l}$
being the Fourier coefficients of a measure.
\begin{prop}
\label{prop:measure}Let $\mu$ be a measure on $[0,2\pi]$ with $\widehat{\mu}(l)\to0$
and let $n_{k}$ be a series such that
\[
\lim_{k\to\infty}\sum_{l=-n_{k}}^{n_{k}}\widehat{\mu}(l)e^{ilx}=0\qquad\forall x\in[0,2\pi].
\]
Then $\mu=0$.
\end{prop}

\begin{proof}
Denote $S_{n}(x)=\sum_{l=-n}^{n}\widehat{\mu}(l)e^{ilx}$. For every
$x\in\supp\mu$ there exists an $M(x)$ such that $|S_{n_{k}}(x)|\le M(x)$
for all $k$ (certainly $M$ exists also for $x\not\in\supp\mu$ but
we will not need it). By the Baire category theorem there is an interval
$I$ and a value $M$ such that the set $\{x:M(x)\le M\}$ is dense
in $I\cap\supp\mu$ (and $I\cap\supp\mu\ne\emptyset$). Note that
we are using the Baire category theorem on the support of $\mu$,
which is compact (here and below support will always mean in the distributional
sense, and in particular will be compact). By continuity, in fact
$M(x)\le M$ for all $x\in I\cap\supp\mu$. Let $\varphi$ be a smooth
function supported on $I$ (and $\varphi(x)\ne0$ for all $x\in I^{\circ}$).
We apply the localisation principle (see, e.g.\ \cite[theorem IX.4.9]{Z})
and get that the series
\[
\varphi(x)\sum\widehat{\mu}(l)e^{ilx}\qquad\text{and}\qquad\sum\widehat{\varphi\mu}(l)e^{ilx}
\]
are uniformly equiconvergent. Hence $\varphi\mu$ satisfies the same
property as $\mu$ i.e.
\[
\lim_{k\to\infty}\sum_{l=-n_{k}}^{n_{k}}\widehat{\varphi\mu}(l)e^{ilx}=0\qquad\forall x\in[0,2\pi]
\]
and further, this convergence is bounded on $\supp\varphi\mu$ (since
$\supp\varphi\mu=I\cap\supp\mu$). If $\mu\ne0$ then also $\varphi\mu\ne0$.

The conclusion of the previous paragraph is that we could have, without
loss of generality, assumed to start with that $S_{n_{k}}$ is bounded
on $\supp\mu$. Let us therefore make this assumption (so we do not
have to carry around the notation $\varphi$). We now argue as follows:
\[
\sum_{l=-n_{k}}^{n_{k}}|\widehat{\mu}(l)|^{2}=\sum_{l=-\infty}^{\infty}\overline{\widehat{S_{n_{k}}}(l)}\cdot\widehat{\mu}(l)=\int\overline{S_{n_{k}}}(x)\,d\mu(x)\le M||\mu||
\]
where the second equality is due to Parseval and where $M$ is again
the maximum of $|S_{n_{k}}|$ on $\supp\mu$. Since this holds for
all $k$, we get that $\sum|\widehat{\mu}(l)|^{2}<\infty$, so $\mu$
is in fact an $L^{2}$ function. But this is clearly impossible, since
the Fourier series of an $L^{2}$ function converges in measure to
it.
\end{proof}
The crux of the proof is that $S_{n_{k}}$ is small where $\mu$ is
supported. The proofs of theorems \ref{thm:doubly exponentially}
and \ref{thm:dim} replace $\mu$ with a different partial sum, $S_{s}$
for some carefully chosen $s$ (roughly, for $s\approx n_{k}^{3/2}$)
and show that $S_{n_{k}}$ is small where $S_{s}$ is essentially
supported. The details are below.

Let us remark that the only place where the condition $\widehat{\mu}(l)\to0$
was used in the proof of proposition \ref{prop:measure} is in the
application of the localisation principle. This can be circumvented,
with a slightly more involved argument. See details in \S \ref{sec:Localisation}.
Similarly theorems \ref{thm:doubly exponentially} and \ref{thm:dim}
may be generalised from $c_{l}\to0$ to $c_{l}$ bounded, at the expense
of a more involved use of the localisation principle. 

\section{Construction}

It will be convenient to work in the interval $[0,1]$ and not carry
around $\pi$-s, so define 
\[
e(x)=e^{2\pi ix}.
\]
For an integrable function $f$ we define the usual Fourier partial
sums,
\[
S_{n}(f;x)=\sum_{l=-n}^{n}\widehat{f}(l)e(lx).
\]
In this paper ``smooth'' means $C^{2}$, but the proofs work equally
well with higher smoothness (up to the quasianalytic threshold). We
use $C$ and $c$ to denote arbitrary constants, whose value might
change from line to line or even inside the same line. We use $C$
for constants which are large enough, and $c$ for constants which
are small enough. We use $||\cdot||$ for the $L^{2}$ (or $\ell^{2}$)
norm, other $L^{p}$ norms are denoted by $||\cdot||_{p}$ (except
one place in the introduction where we used $||\mu||$ for the norm
of the measure $\mu$). For a set $E\subset[0,1]$ we denote by $|E|$
the Lebesgue measure of $E$.

\subsection{The localisation principle\label{subsec:Rajchman}}

Let us recall Riemann's localisation principle: as formulated by Riemann,
it states that the convergence of a trigonometric series at a point
$x$ depends only on the behaviour of the Riemann function at a neighbourhood
of $x$. See \cite[\S IX.4]{Z}. Rajchman found a formulation of the
principle which does not use the Riemann function and has a simple
proof. It states that for any $c_{l}\to0$ and any smooth function
$\varphi$,
\begin{equation}
\varphi(x)\sum c_{l}e(lx)\text{ and }\sum(c*\widehat{\varphi})(l)e(lx)\text{ are uniformly equiconvergent}\label{eq:Rajchman}
\end{equation}
where $c*\widehat{\varphi}$ is a discrete convolution. See \cite[theorem IX.4.9]{Z},
or the proof of theorem \ref{thm:local} below, which follows Rajchman's
approach precisely. We will use Rajchman's theorem both on and off
the support of $\sum c_{l}e(lx)$ (denote this support by $K$). Off
$K$, it has the following nice formulation: if $c_{l}\to0$ then
\begin{equation}
\sum c_{l}e(lx)=0\qquad\forall x\not\in K.\label{eq:KS}
\end{equation}
and further, convergence is uniform on any closed interval disjoint
from $K$. To the best of our knowledge, this precise formulation
first appeared in \cite[Proposition 1, \S V.3, page 54]{KS94}.

\subsection{First estimates}
\begin{lem}
\label{lem:willthisbetheendofthisproblematlast}For every $\epsilon>0$
there exists a smooth function $u:[0,1]\to\mathbb{R}$ with $u(0)=u(1)=0$,
$u(x)\in[0,1]$, and $||\widehat{u-1}||_{\infty}<\epsilon$.
\end{lem}

When we say that $u$ is smooth we mean also when extended periodically
(or when extended by $0$, which is the same under the conditions
above).
\begin{proof}
Take any standard construction of a smooth function satisfying $u(0)=u(1)=0$,
$u(x)\in[0,1]$ and $u(x)=1$ for all $x\in[\frac{1}{2}\epsilon,1-\frac{1}{2}\epsilon]$.
The condition on the Fourier coefficients then follows by $||\widehat{u-1}||_{\infty}\le||u-1||_{1}$.
\end{proof}
\begin{lem}
\label{lem:vandermonde}For every $\epsilon>0$ there exists a smooth
function $h:[0,1]\to\mathbb{R}$ and an $n\in\mathbb{N}$ such that
\begin{enumerate}
\item $\widehat{h}(0)=1$
\item $\supp h\subset[0,\frac{1}{2}]$
\item For all $x\in[0,\frac{1}{2}]$, $|S_{n}(h;x)|<\epsilon$.
\end{enumerate}
\end{lem}

\begin{proof}
Let $P$ be an arbitrary trigonometric polynomial satisfying that
$\widehat{P}(0)=1$ and $|P(x)|<\epsilon$ for all $x\in[0,\frac{1}{2}]$.
Let $n=\deg P$, let $m=2n+1$ and let $q$ be a smooth function supported
on $[0,\nicefrac{1}{2m}]$ with $\widehat{q}(k)\ne0$ for all $|k|\le n$.
Examine a function $h$ of the type
\[
h(x)=\sum_{j=0}^{m-1}a_{j}q\Big(x-\frac{j}{2m}\Big).
\]
Then $h$ is smooth, supported on $[0,\frac{1}{2}]$, and its Fourier
coefficients are given by
\[
\widehat{h}(k)=\widehat{q}(k)\sum_{j=0}^{m-1}a_{j}e(-jk/2m).
\]
The matrix $\{e(-jk/2m):j\in\{0,\dotsc,m-1\},k\in\{-n,\dotsc,n\}\}$
is a Vandermonde matrix hence invertible, so one may find $a_{j}$
such that $\sum a_{j}e(-jk/2m)\linebreak[0]=\widehat{P}(k)/\widehat{q}(k)$
for all $k\in\{-n,\dotsc,n\}$. With these $a_{j}$ our $h$ satisfies
$\widehat{h}(k)=\widehat{P}(k)$ for all $k$ such that $|k|\le n$
so $S_{n}(h)=P$ which has the required properties.
\end{proof}
\begin{rem*}
The coefficients of the $h$ given by lemma \ref{lem:vandermonde}
are typically large. The reason is the Vandermonde matrix applied.
We need to invert the Vandermonde matrix and its inverse has a large
norm, exponential in $n$ (the inverse of a Vandermonde matrix has
an explicit formula). To counterbalance this last sentence a little,
let us remark that $n$, the degree of the polynomial $P$ used during
the proof can be taken to be logarithmic in $\epsilon$. This requires
to choose a good $P$. For this purpose we apply the following theorem
of Szeg\H{o}: for every compact $K\subset\mathbb{C}$ there exists
monic polynomials $Q_{n}$ with $\max_{x\in K}|Q_{n}(x)|=(\capa(K)+o(1))^{n}$.
See \cite[corollary 5.5.5]{R95}. We apply Szeg\H{o}'s theorem with
$K=\{e(x):x\in[0,\nicefrac{1}{2}]\}$ and then define $P_{n}(x)=\textrm{Re}(e(-nx)Q_{n}(e(x)))$.
We get that $\widehat{P_{n}}(0)=1$ and $\max_{x\in[0,1/2]}|P_{n}(x)|\le(\capa(K)+o(1))^{n}$.
The capacity of $K$ can be calculated by writing explicitly a Riemann
mapping between $\mathbb{C}\setminus K$ and $\{z:|z|>1\}$ and is
$\nicefrac{1}{\sqrt{2}}$, and in particular smaller than 1 (see \cite[theorem 5.2.3]{R95}
for the connection to Riemann mappings). Hence it is enough to take
$n=C\log\nicefrac{1}{\epsilon}$ to ensure that $P$ would satisfy
$|P(x)|\le\epsilon$ for all $x\in[0,\nicefrac{1}{2}]$. With this
$P$ the norm of $h$ would be polynomial in $\epsilon$.
\end{rem*}

\subsection{Reducing the coefficients}

In the next lemma we reduce the Fourier coefficients using a method
inspired by a proof of the Menshov representation theorem (see \cite{O85}).
We separate the interval $[0,1]$ into many small pieces and on each
put a copy of the $h$ above, scaled differently. Unlike in typical
applications of Menshov's approach, we do not have each copy of $h$
sit in a distinct ``spectral interval'' but they are rather intertwined.
The details are below. Still, like in other applications of Menshov's
technique, the resulting set is divided into many small intervals
in a way that pushes the dimension up. This is why we are unable to
construct an example supported on a set with dimension less than $1$.
\begin{lem}
\label{lem:no g}For every $\epsilon>0$ there exists a smooth function
$f:[0,1]\to\mathbb{R}$ and an $n\in\mathbb{N}$ with the following
properties:
\begin{enumerate}
\item $\widehat{f}(0)=1$.
\item For all $k\ne0$, $|\widehat{f}(k)|<\epsilon$.
\item For every $x\in\supp f,$ $|S_{n}(f;x)|<\epsilon$.
\end{enumerate}
\end{lem}

\begin{proof}
We may assume without loss of generality that $\epsilon<\frac{1}{2}$,
and it is enough to replace requirement (i) by the weaker requirement
$|\widehat{f}(0)-1|<\epsilon$ (and then normalise).

\emph{1. }Let $h$ be the function given by lemma \ref{lem:vandermonde}
with $\epsilon_{\text{lemma \ref{lem:vandermonde}}}=\epsilon/4$,
and denote $m=n_{\text{lemma \ref{lem:vandermonde}}}$. In other words,
$h$ satisfies
\begin{gather*}
\widehat{h}(0)=1\qquad\qquad\supp h\subset[0,\tfrac{1}{2}]\\
|S_{m}(h;x)|<\tfrac{1}{4}\epsilon\quad\forall x\in[0,\tfrac{1}{2}].
\end{gather*}
Let $a>2||h||_{1}$/$\epsilon$ be some integer. Let $u$ be the
function given by lemma \ref{lem:willthisbetheendofthisproblematlast}
with $\epsilon_{\textrm{lemma \ref{lem:willthisbetheendofthisproblematlast}}}=\epsilon/2$
i.e.\ $u$ is smooth from $[0,1]$ to $[0,1]$, $u(0)=u(1)=0$ and
$u$ satisfies 
\[
||\widehat{u-1}||_{\infty}<{\textstyle \frac{1}{2}}\epsilon.
\]
Let $v(x)=u(xa)$ (extended to zero outside $[0,1/a]$). Let $r$
be a large integer parameter to be fixed later, depending on all previously
defined quantities ($\epsilon$, $h$, $m$, $a$ and $u$). Define
\[
f(x)=\sum_{j=0}^{a-1}v\Big(x-\frac{j}{a}\Big)h(x(r^{3}+jr)).
\]
The role of the quantities $r^{3}+jr$ will become evident later. 

Let us see that $f$ satisfies all required properties. It will be
easier to consider trigonometric polynomials rather than smooth functions
so define
\begin{gather}
\begin{aligned}H & \colonqq S_{\lfloor r/2\rfloor-1}(h)\qquad\qquad\qquad & V & \colonqq S_{\lfloor r/2\rfloor-1}(v)\end{aligned}
\nonumber \\
F(x)\colonqq\sum_{j=0}^{a-1}V\Big(x-\frac{j}{a}\Big)H(x(r^{3}+jr)).\label{eq:def f'}
\end{gather}
The smoothness of $v$ and $h$ imply that $||\widehat{v-V}||_{1}$
and $||\widehat{h-H}||_{1}$ can be taken arbitrarily small as $r\to\infty$.
Since
\[
||\widehat{f-F}||_{1}\le\sum_{j=0}^{a-1}||\widehat{v-V}||_{1}||\widehat{h}||_{1}+||\widehat{V}||_{1}||\widehat{h-H}||_{1}
\]
we may take $r$ sufficiently large and get $||\widehat{f-F}||_{1}<\frac{1}{2}\epsilon$
(but do not fix the value of $r$ yet). Thus, with such an $r$, we
need only show
\begin{enumerate}
\item $||\widehat{F-1}||_{\infty}<\frac{1}{2}\epsilon$
\item For every $x\in\supp f$, $|S_{n}(F;x)|<\frac{1}{2}\epsilon$ (note
that we take $x$ in $\supp f$ and not in $\supp F$).
\end{enumerate}
\emph{2}. We start with the estimate of $\widehat{F-1}$. Examine
one summand in the definition of $F$, (\ref{eq:def f'}). Denoting
$G_{j}=V(x-j/a)H(x(r^{3}+jr))$ we have
\begin{equation}
\widehat{G_{j}}(l)=\begin{cases}
\widehat{V}(p)\widehat{H}(q)e(-pj/a) & l=p+q(r^{3}+jr),\;|p|,|q|<r/2\\
0 & \text{otherwise}.
\end{cases}\label{eq:gj hat}
\end{equation}
In particular, $l$ and $j$ determine $p$ and $q$ uniquely. An
immediate corollary is:
\begin{equation}
||\widehat{G_{j}}||_{\infty}=||\widehat{V}||_{\infty}||\widehat{H}||_{\infty}\le||v||_{1}||h||_{1}\le\frac{||h||_{1}}{a}<\frac{\epsilon}{2}\label{eq:gj hat infty}
\end{equation}
where the last inequality is from the definition of $a$. Assume now
that $r>a$. Then we can extract another corollary from (\ref{eq:gj hat}):
that the different $G_{j}$ have disjoint spectra, except at $(-r/2,r/2)$.
Hence 
\begin{equation}
|\widehat{F}(l)|=\max_{j}|\widehat{G_{j}}(l)|\stackrel{\textrm{(\ref{eq:gj hat infty})}}{<}\frac{\epsilon}{2}\qquad\forall|l|\ge r/2.\label{eq:l>r/2}
\end{equation}
Finally, for $l\in(-r/2,r/2)$ we have that $F$ ``restricted spectrally
to $(-r/2,\linebreak[0]r/2)$'' is simply $\sum_{j}V(x-j/a)$ so
its Fourier spectrum is simply that of $u$ spread out. Since $||\widehat{u-1}||<\frac{\epsilon}{2}$
we get also in this case $|\widehat{F-1}(l)|<\frac{1}{2}\epsilon$.
For those who prefer formulas, just note in (\ref{eq:gj hat}) that
if $l\in(-r/2,r/2)$ then $q=0$ and since $\widehat{H}(0)=1$ we
get
\[
\widehat{F}(l)=\sum_{j=0}^{a-1}\widehat{V}(l)e(-lj/a)=\begin{cases}
a\widehat{V}(l) & l\equiv0\text{ mod }a\\
0 & \text{otherwise}.
\end{cases}
\]
Recall that $v(x)=u(xa)$ so for $l\equiv0$ mod $a$ we have 
\[
|\widehat{aV-1}(l)|\le|\widehat{av-1}(l)|=|\widehat{u-1}(l/a)|<\tfrac{1}{2}\epsilon.
\]
With (\ref{eq:l>r/2}) we get $||\widehat{F-1}||_{\infty}<\frac{1}{2}\epsilon$,
as needed.

\emph{3}. Finally, we need to define $n$ and see that $S_{n}(F)$
is small on $\supp f$. Assume $r>m$ and define 
\[
n=m(r^{3}+r^{2}).
\]
This value of $n$ has the property that
\begin{align*}
n & >m(r^{3}+jr)+r/2\\
n & <(m+1)(r^{3}+jr)-r/2
\end{align*}
for all $j\in\{0,\dotsc,a-1\}$. We now see why it was important to
choose the spacings of the arithmetic progressions to be $r^{3}+jr$:
these spacings need to be different to have separation of the spectra
of the different $G_{j}$ (and they must be different by at least
$r$, because the spectra of the $G_{j}$ are arranged in blocks of
size $r$), but they need to be sufficiently close that it would still
be possible to ``squeeze'' an $n$ between all the terms that correspond
to the $m^{\textrm{th}}$ block in all $G_{j}$ and all the terms
that correspond to the $m+1^{\textrm{st}}$ blocks. The $r^{3}$ in
the spacings ensures that.

Using (\ref{eq:gj hat}) gives that 
\[
S_{n}(G_{j};x)=S_{m}\big(H;x(r^{3}+jr)\big)\cdot V\Big(x-\frac{j}{a}\Big).
\]
At this point it will be easier to compare to $v$ rather than to
$V$, so write
\[
S_{n}(G_{j};x)=S_{m}\big(H;x(r^{3}+jr)\big)\cdot v\Big(x-\frac{j}{a}\Big)+E_{j}
\]
and note that for $r$ sufficiently large $E_{j}$ can be taken to
be arbitrarily small. Take $r$ so large as to have
\begin{equation}
\bigg|S_{n}(F;x)-\sum_{j=0}^{a-1}S_{m}(H;x(r^{3}+jr))v\Big(x-\frac{j}{a}\Big)\bigg|<\tfrac{1}{4}\epsilon\qquad\forall x\in[0,1].\label{eq:h tag u notag}
\end{equation}
This is our last requirement from $r$ and we may fix its value now. 

For every $x\in[0,1]$ there is at most one $j_{0}$ such that $v(x-j_{0}/a)\ne0$,
namely $j_{0}=\lfloor x/a\rfloor$. If $x\in\supp f$ then it must
be the case that $x(r^{3}+j_{0}r)\in[0,\frac{1}{2}]$ mod 1. But in
this case, by our definition, 
\[
|S_{m}(H;x(r^{3}+j_{0}r))|<\tfrac{1}{4}\epsilon.
\]
We get
\[
x\in\supp f\implies\bigg|\sum_{j=0}^{a-1}S_{m}(H;x(r^{3}+jr))\cdot v\Big(x-\frac{j}{a}\Big)\bigg|<\tfrac{1}{4}\epsilon,
\]
and with (\ref{eq:h tag u notag}) we get $|S_{n}(F;x)|<\frac{1}{2}\epsilon$,
as needed.
\end{proof}
\begin{lem}
\label{lem:yes g}Let $f:[0,1]\to\mathbb{R}$ be smooth, $\epsilon>0$
and $N\in\mathbb{N}$. Then there exists a smooth function $g:[0,1]\to\mathbb{R}$
satisfying
\begin{enumerate}
\item $\supp g\subseteq\supp f$.
\item \label{enu:g-f^}For all $n\in\mathbb{Z}$, $|\widehat{g}(n)-\widehat{f}(n)|<\epsilon$
\item \label{enu:Sng small}For some $n>N$ we have 
\[
|S_{n}(g;x)|<\epsilon\qquad\forall x\in\supp g.
\]
\end{enumerate}
\end{lem}

\begin{proof}
Let $h$ be the function from lemma \ref{lem:no g} with $\epsilon_{\text{lemma \ref{lem:no g}}}=\epsilon/2||\widehat{f}||_{1}$.
Denote by $m$ the integer output of lemma \ref{lem:no g} i.e.\ the
number such that $S_{m}(h;x)<\epsilon/(2||\widehat{f}||_{1})$ for
all $x\in\supp h$. Let $r$ be large enough so that
\[
\sum_{|k|\ge r/2}|\widehat{f}(k)|<\epsilon/(2||\widehat{h}||_{1})
\]
and such that $r(m+1/2)>N$ (let $r$ be even). Denote
\begin{align*}
g(x) & \colonqq f(x)h(rx)\qquad n\colonqq r(m+1/2)
\end{align*}
where $h$ is extended periodically to $\mathbb{R}$. Let us see that
$g$ and $n$ satisfy the requirements of the lemma. The smoothness
of $g$ follows from those of $f$ and $h$. Condition (\ref{enu:g-f^})
follows because 
\[
\widehat{g}(k)-\widehat{f}(k)=\sum_{l}\widehat{h-1}(l)\widehat{f}(k-lr)
\]
and because $||\widehat{h-1}||_{\infty}\le\epsilon/(2||\widehat{f}||_{1})$.
Finally, to see condition (\ref{enu:Sng small}) write
\[
F\colonqq S_{r/2}(f)\qquad G(x)\colonqq F(x)h(rx)
\]
and note that $||\widehat{g-G}||_{1}\le||\widehat{f-F}||_{1}||\widehat{h}||_{1}<\frac{1}{2}\epsilon$.
To estimate $S_{n}(G)$, note that if $x\in\supp g$ then $rx\in\supp h$
mod 1 and hence $S_{m}(h;rx)<\epsilon/(2||\widehat{f}||_{1})$. But
\[
S_{n}(G;x)=F(x)S_{m}(h;rx)
\]
but since $|F(x)|\le||\widehat{F}||_{1}\le||\widehat{f}||_{1}$ we
get
\[
|S_{n}(G;x)|\le||\widehat{f}||_{1}\frac{\epsilon}{2||\widehat{f}||_{1}}=\frac{\epsilon}{2}
\]
finishing the lemma.
\end{proof}

\subsection{Proof of theorem \ref{thm:example}}

The coefficients $c_{l}$ will be constructed by inductively applying
lemma \ref{lem:yes g}. Define therefore $f_{1}=1$ and $n_{1}=2$,
and for all $k\ge1$ define $f_{k+1}=g_{\text{lemma \ref{lem:yes g}}}$
and $n_{k+1}=n_{\text{lemma \ref{lem:yes g}}}$ where lemma \ref{lem:yes g}
is applied with $f_{\text{lemma \ref{lem:yes g}}}=f_{k}$, $\epsilon_{\text{lemma \ref{lem:yes g}}}=2^{-k}/n_{k}$
and $N_{\text{lemma \ref{lem:yes g}}}=n_{k}+1$ (this last parameter
merely ensures that the $n_{k}$ are increasing). We now claim that
$\widehat{f_{k}}(l)$ converges as $k\to\infty$, and that the limit,
$c_{l}$, satisfies the requirements of the theorem.

The fact that $\lim_{k\to\infty}\widehat{f_{k}}(l)$ exists is clear,
because $\widehat{f_{k+1}}(l)-\widehat{f_{k}}(l)<2^{-k}/n_{k}$. Denote
\[
c_{l}=\lim_{k\to\infty}\widehat{f_{k}}(l).
\]
This also shows that $c_{l}\to0$. 

Denote now $S_{n}=\sum_{l=-n}^{n}c_{l}e(lx).$ To see that $S_{n_{k}}(x)\to0$
for all $x$ we separate into $x\in\cap\supp f_{k}$ and the rest.
Note that $\cap\supp f_{k}$ contains the support of the distribution
$\delta:=\sum c_{l}e(lx)$. Indeed, if $\varphi$ is a Schwartz test
function supported outside $\cap\supp f_{k}$ then $\supp\varphi\cap\supp f_{k}$
is a sequence of compact sets decreasing to the empty set (recall
that $\supp f_{k+1}\subseteq\supp f_{k}$) so for some finite $k_{0}$
we already have $\supp\varphi\cap\supp f_{k}=\emptyset$ for all $k>k_{0}$.
This of course implies that $\langle\varphi,f_{k}\rangle=0$. Taking
the limit $k\to\infty$ we get $\langle\varphi,\delta\rangle=0$ (we
may take the limit since $||\widehat{f_{k}-\delta}||_{\infty}\to0$
while $\widehat{\varphi}\in l_{1}$). Since this holds for any $\varphi$
supported outside $\cap\supp f_{k}$ we get $\supp\delta\subset\cap\supp f_{k}$,
as claimed.

Now, for $x\not\in\cap\supp f_{k}$ we use the localisation principle
in the form (\ref{eq:KS}) and get  
\begin{equation}
\lim_{n\to\infty}S_{n}(x)=0\qquad\forall x\not\in\bigcap\supp f_{k}\label{eq:outside support}
\end{equation}
i.e.\ outside the support it is not necessary to take a subsequence.

Finally, examine $x\in\supp f_{k}$. By clause (\ref{enu:Sng small})
of lemma \ref{lem:yes g}
\begin{equation}
|S_{n_{k}}(f_{k};x)|<\frac{1}{2^{k-1}n_{k-1}}.\label{eq:nk at k}
\end{equation}
For any $j\ge k$, the condition $|\widehat{f_{j+1}}(k)-\widehat{f_{j}}(k)|<2^{-j}/n_{j}\le2^{-j}/n_{k}$
means that
\[
|S_{n_{k}}(f_{j+1};x)-S_{n_{k}}(f_{j};x)|<3\cdot2^{-j}
\]
which we sum (also with (\ref{eq:nk at k})) to get
\[
|S_{n_{k}}(f_{j};x)|<8\cdot2^{-k}\qquad\forall j\ge k
\]
and taking limit as $j\to\infty$ gives
\[
\big|S_{n_{k}}(x)\big|<8\cdot2^{-k}\qquad\forall x\in\supp f_{k}.
\]
We conclude
\[
\lim_{k\to\infty}S_{n_{k}}(x)=0\qquad\forall x\in\bigcap\supp f_{k}.
\]
With (\ref{eq:outside support}), the theorem is proved.\qed
\begin{rem*}
The observant reader probably noticed that we use smooth functions
as our building blocks rather than trigonometric polynomials, and
hence our construction does not naturally have large spectral gaps,
unlike many constructions of null series. This is not a coincidence:
it is not possible to have many large spectral gaps in any series
that satisfies the requirements of Theorem \ref{thm:example}. Precisely,
a theorem of Beurling states that any tempered distribution $\sum c_{l}e^{ilt}$
whose supported is not the whole interval (and our $c_{l}$ satisfy
that, see Lemma \ref{lem:K} below) cannot have $c_{l}=0$ on an increasing
sequence of intervals $[a_{k},b_{k}]$ satisfying $\sum(b_{k}-a_{k})^{2}/a_{k}^{2}=\infty$.
See e.g.~\cite[Theorem 4]{B84}.
\end{rem*}

\section{Proof of theorems \ref{thm:doubly exponentially} and \ref{thm:dim}}

The following lemma summarises some properties of the support of the
distribution.
\begin{lem}
\label{lem:K}Let $c_{l}\to0$ and $n_{k}\to\infty$ such that
\[
\lim_{k\to\infty}S_{n_{k}}(x)=0\qquad\forall x\qquad S_{n}(x)=\sum_{l=-n}^{n}c_{l}e(lx)
\]
Let $K$ be the support of the distribution $\sum c_{l}e(lx)$. Then
\begin{enumerate}
\item \label{enu:critical}$K=\{x:\forall\epsilon>0,S_{n_{k}}\text{ is unbounded in }(x-\epsilon,x+\epsilon)\}$.
\item \label{enu:nowhere dense}$K$ is nowhere dense.
\end{enumerate}
\end{lem}

\begin{proof}
We start with clause (\ref{enu:critical}). On the one hand, if $x\not\in K$
then the localisation principle (\ref{eq:KS}) tells us that $S_{n}\to0$
uniformly in some neighbourhood of $x$. On the other hand, if $S_{n_{k}}$
is bounded in some neighbourhood $I$ of $x$ then for any smooth
test function $\varphi$ supported on $I$ we have 
\[
\langle\varphi,\sum c_{l}e(lx)\rangle=\sum_{l=-\infty}^{\infty}c_{l}\widehat{\varphi}(l)=\lim_{k\to\infty}\sum_{l=-n_{k}}^{n_{k}}c_{l}\widehat{\varphi}(l)=\lim_{k\to\infty}\int\varphi S_{n_{k}}
\]
but the integral on the right-hand side tends to zero from the bounded
convergence theorem. This shows (\ref{enu:critical}).

To see clause (\ref{enu:nowhere dense}) examine the function $N(x)=\sup_{k}|S_{n_{k}}(x)|$
and apply the Baire category theorem to the sets $\{x:N(x)\ge M\}$
for all integer $M$. We get, in every interval $I$, an open interval
$J\subset I$ and an $M$ such that $N(x)\le M$ on a dense subset
of $J$. continuity shows that in fact $N(x)\le M$ on all of $J$
and hence $J\cap K=\emptyset$, as needed.
\end{proof}
\begin{rem*}
Without the condition $c_{l}\to0$ it still holds that
\[
K\subset\{x:\forall\epsilon>0,S_{n_{k}}\text{ is unbounded in }(x-\epsilon,x+\epsilon)\}
\]
and that $K$ is nowhere dense. The proof is the same.
\end{rem*}
We will now make a few assumptions that will make the proof less cumbersome.
First we assume that $c_{-l}=\overline{c_{l}}$ (or, equivalently,
that the $S_{n}$ are real). It is straightforward to check that this
assumption may be made without loss of generality in both theorems
\ref{thm:doubly exponentially} and \ref{thm:dim}. Our next assumption
is:

\begin{assumption}In the next lemma we assume that $S_{n_{k}}$ is
bounded on $K,$ the support of the distribution $\sum c_{l}e(lx)$.
Further, whenever we write ``$C$'', the constant is allowed to
depend on $\sup\{|S_{n_{k}}(x)|:x\in K,k\}$. \end{assumption} 

As in the proof of proposition \ref{prop:measure}, we will eventually
remove this assumption by a simple localisation argument.
\begin{lem}
\label{lem:new riemann}Let $c_{l}$, $n_{k}$ and $S_{n}$ be as
in the previous lemma. Let $r$ be a sufficiently large number in
our sequence (i.e.\ $r=n_{k}$ for some $k$) and let $s>r^{3/2}\log^{4}r$
not necessarily in the sequence. Then
\begin{equation}
||S_{s}||\ge c||S_{r}||^{2}.\label{eq:no dim}
\end{equation}
\end{lem}

Lemma \ref{lem:new riemann} is used in the proof of theorem \ref{thm:doubly exponentially}.
We will also need a version of lemma \ref{lem:new riemann} for theorem
\ref{thm:dim} but that version is somewhat clumsy to state, so rather
than doing it now, we postpone it to the end of the proof of the lemma,
the impatient can jump to (\ref{eq:dim-2}) to see it. The only point
worthy of making now is that we will need a result that holds for
all $s>r$ so throughout the proof of lemma \ref{lem:new riemann}
we will note when we use the assumption $s>r^{3/2}\log^{4}r$ and
when $s>r$ is enough.

It might be tempting to think that lemma \ref{lem:new riemann} is
a lemma on trigonometric polynomials, i.e.\ that it would have been
possible to simply formulate it for $S_{r}$ being the Fourier partial
sum of $S_{s}$. However, as the proof will show, we need to have
the full distribution acting ``in the background'' restricting both
what $S_{r}$ and $S_{s}$ may do.
\begin{proof}
 Fix $r$ and $s>r$. It will be convenient to assume $s/\log^{4}s\ge r$,
so let us make this assumption until further notice. Denote $K=\supp\sum c_{l}e(lx)$
and let $I$ be a component of $K^{c}$ with $|I|>(2\log^{3}s)/s$.
Let $\varphi_{I}$ be a function with the following properties:
\begin{enumerate}
\item If $I=(a,b)$ then $\varphi_{I}$ restricted to $[a+(\log^{3}s)/s,b-(\log^{3}s)/s]$
is identically $1$.
\item $\supp\varphi_{I}\subset I$ (note that $I$ is open, so this inclusion
must be strict).
\item $\varphi_{I}(x)\in[0,1]$ for all $x\in[0,1]$.
\item $|\widehat{\varphi_{I}}(l)|\le C\exp\Big(-c\sqrt{(|l|\log^{3}s)/s}\Big)$.
\end{enumerate}
It is easy to see that such a $\varphi_{I}$ exists \textemdash{}
take a standard construction of a $C^{\infty}$ function $\psi:\mathbb{R}\to[0,1]$
with $\psi|_{(-\infty,0)}\equiv0$, $\psi|_{[1,\infty)}\equiv1$ and
$||\psi^{(k)}||_{\infty}\le C(k!)^{2}$ (see e.g.\ \cite[\S V.2]{K04}),
define $\varphi$ by mapping $\psi$ (restricted to an appropriate
interval) linearly to each half of $I$ and estimate $\widehat{\varphi}(l)$
by writing $|\widehat{\varphi}(l)|\le l^{-k}\cdot||\varphi^{(k)}||_{\infty}$
and optimising over $k$. We skip any further details.

Let
\[
\varphi=\sum_{I}\varphi_{I}
\]
where the sum is taken over all $I$ as above, i.e.\ $I$ is a component
of $K^{c}$ with $|I|>(2\log^{3}s)/s$. Our lemma is based on the
following decomposition
\[
||S_{r}||^{2}=\int S_{s}\cdot S_{r}=\int S_{s}\cdot S_{r}\cdot\varphi+\int S_{s}\cdot S_{r}\cdot(1-\varphi).
\]
To estimate the first summand, first note that
\begin{align*}
|\widehat{S_{r}\cdot\varphi_{I}}(n)| & \le\sum_{l=-r}^{r}|c_{l}\widehat{\varphi_{I}}(n-l)|\le C\sum_{l=-r}^{r}\exp\left(-c\sqrt{\frac{|n-l|\log^{3}s}{s}}\right)\\
 & \stackrel{\mathclap{{(*)}}}{\le}Cr\exp\Big(-c\sqrt{(|n|\log^{3}s)/s}\Big).
\end{align*}
The inequality marked by $(*)$ is a simple exercise, but let us remark
on it anyway. If $|n|<2s/\log^{3}s$ then both sides of $(*)$ are
$\approx r$ and it holds. If $|n|\ge2s/\log^{3}s$ then, because
we assumed $s/\log^{4}s>r$, we get that $\frac{1}{2}|n|>r\ge|l|$
so $|n-l|\ge\frac{1}{2}|n|$ and $(*)$ holds again. 

Summing over $I$ gives
\[
|\widehat{S_{r}\cdot\varphi}(n)|\le Crs\exp\Big(-c\sqrt{(|n|\log^{3}s)/s}\Big).
\]
Next, because $S_{r}\cdot\varphi$ is supported outside $K$ we have
\[
\sum_{l=-\infty}^{\infty}c_{l}\widehat{S_{r}\cdot\varphi}(l)=0
\]
so
\[
\int S_{s}\cdot S_{r}\cdot\varphi=-\sum_{|l|>s}c_{l}\widehat{S_{r}\cdot\varphi}(l)
\]
and then
\begin{align}
\Big|\int S_{s}\cdot S_{r}\cdot\varphi\Big| & \le\sum_{|l|>s}|c_{l}|\cdot|\widehat{S_{r}\cdot\varphi}(l)|\le C\sum_{|l|>s}rs\exp\Big(-c\sqrt{(|l|\log^{3}s)/s}\Big)\nonumber \\
 & \le C\exp(-c\log^{3/2}s)\label{eq:SrSsphi}
\end{align}
which is negligible (the last inequality can be seen, say, by dividing
into blocks of size $s$, getting the expression $Crs^{2}\sum_{k=1}^{\infty}\exp(-ck\log^{3/2}s)$
which is clearly comparable to its first term $\exp(-c\log^{3/2}s)$,
and finally noting that the term $rs^{2}\le s^{3}$ may be dropped
at the price of changing the constants $C$ and $c$).

We move to the main term, $\int S_{r}S_{s}(1-\varphi)$, which we
will estimate using Cauchy-Schwarz
\[
\Big|\int S_{s}\cdot S_{r}\cdot(1-\varphi)\Big|\le||S_{s}||\cdot||S_{r}(1-\varphi)||.
\]
Hence we need to estimate $||S_{r}(1-\varphi)||$. For this we do
not need the smoothness of $\varphi$ so define $E:=\supp(1-\varphi)$
and replace $1-\varphi$ with $\mathbbm{1}_{E}$. Thus the lemma will
be proved once we show
\begin{claim*}
If $s>r^{3/2}\log^{4}r$ then $||S_{r}\mathbbm{1}_{E}||\le C$.
\end{claim*}
To show the claim, we need the following definition. Let $I$ be a
component of $K^{c}$ (not necessarily large, any component) and denote,
for each such $I$ and for each $M$,
\[
A_{I,M}\colonqq|\{x\in I\cap E:|S_{r}'|\in[M,2M]\}|.
\]
We need a simple bound for the values of $M$ that interest us, and
we use that $|S_{r}'|\le Cr^{2}$ always (simply because the $c_{l}$
are bounded). For any $x\in I\cap E$ we may then estimate $S_{r}$
itself by integrating $S_{r}'$ from the closest point of $K$ up
to $x$. We get
\begin{equation}
|S_{r}(x)|\le C+\sum_{\substack{M=1\\
\text{scale}
}
}^{Cr^{2}}2MA_{I,M}\label{eq:SraIM-1}
\end{equation}
where the word ``scale'' below the $\Sigma$ means that $M$ runs
through powers of $2$ (i.e.\ it is equivalent to $\sum_{m=0}^{\lfloor\log_{2}Cr^{2}\rfloor}$
and $M=2^{m}$). Note that (\ref{eq:SraIM-1}) uses our assumption
that $\max_{x\in K}|S_{r}(x)|\le C$ for a constant $C$ independent
of $r$ (and the additive constant $C$ in (\ref{eq:SraIM-1}) is
the same $C$). Rewriting (\ref{eq:SraIM-1}) as
\[
|S_{r}\cdot\mathbbm{1}_{E}|\le\sum_{I}\mathbbm{1}_{I\cap E}\Big(C+\sum_{M\textrm{ scale}}^{Cr^{2}}2MA_{I,M}\Big)
\]
gives
\begin{align}
||S_{r}\mathbbm{1}_{E}|| & \le C\Big\Vert\sum_{I}\mathbbm{1}_{I\cap E}\Big\Vert+\sum_{M\textrm{ scale}}^{Cr^{2}}\Big\Vert\sum_{I}2MA_{I,M}\mathbbm{1}_{I\cap E}\Big\Vert\nonumber \\
 & =C\sqrt{|E|}+\sum_{M\textrm{ scale}}^{Cr^{2}}2M\sqrt{\sum_{I}|I\cap E|A_{I,M}^{2}}.\label{eq:norm Sr}
\end{align}
To estimate the sum notice that $A_{I,M}\le|I\cap E|\le2(\log s)^{3}/s$
so
\begin{align*}
\sum_{I}|I\cap E|A_{I,M}^{2} & \le4\frac{\log^{6}s}{s^{2}}\sum_{I}A_{I,M}\le4\frac{\log^{6}s}{s^{2}}|\{x:|S_{r}'(x)|\ge M\}|\\
 & \stackrel{\mathclap{{(*)}}}{\le}4\frac{\log^{6}s}{s^{2}}\frac{||S_{r}'||^{2}}{M^{2}}\le4\frac{\log^{6}s}{s^{2}}\frac{||S_{r}||^{2}r^{2}}{M^{2}}
\end{align*}
where the inequality marked by $(*)$ follows by Chebyshev's inequality.
The sum over scales in (\ref{eq:norm Sr}) has only $C\log r\le C\log s$
terms, so we get
\[
||S_{r}\mathbbm{1}_{E}||\le C\Big(\sqrt{|E|}+\frac{||S_{r}||r\log^{4}s}{s}\Big).
\]
This finishes the claim, since we assumed $s>r^{3/2}\log^{4}r$ and
since $||S_{r}||\le C\sqrt{r}$ because the coefficients $c_{l}$
are bounded. \qed

Let us recall how the claim implies the lemma: using Cauchy-Schwarz
and $(1-\varphi)\le\mathbbm{1}_{E}$ gives 
\begin{equation}
\Big|\int S_{s}\cdot S_{r}\cdot(1-\varphi)\Big|\le C||S_{s}||\Big(\sqrt{|E|}+\frac{||S_{r}||r\log^{4}s}{s}\Big)\label{eq:norm Sr final}
\end{equation}
Recall that (\ref{eq:SrSsphi}) showed that the other term in $||S_{r}||^{2}$
is negligible, so we get the same kind of estimate for $||S_{r}||^{2}$:
\begin{equation}
||S_{r}||^{2}\le C||S_{s}||\bigg(\sqrt{|E|}+\frac{||S_{r}||r\log^{4}s}{s}\bigg).\label{eq:dim}
\end{equation}
With $s>r^{3/2}\log^{4}r$ and $||S_{r}||\le C\sqrt{r}$ equation
(\ref{eq:dim}) translates to $||S_{r}||^{2}\le C||S_{s}||$, as needed.

Before putting the q.e.d.\ tombstone, though, let us reformulate
(\ref{eq:dim}) in a way that will be useful in the proof of theorem
\ref{thm:dim}. We no longer assume $s>r^{3/2}\log^{4}r$ (though
we cannot yet remove the assumption $s/\log^{4}s>r$ from the beginning
of the proof, as it was used to reach (\ref{eq:dim})). Recall that
$E=\supp(1-\varphi)$, that $\varphi=\sum\varphi_{I}$ and that each
$\varphi_{I}$ is $1$ except in a $(\log^{3}s)/s$ neighbourhood
of $K$. Hence $E\subset K+[-(\log s)^{3}/s,(\log s)^{3}/s]$ (the
sum here is the Minkowski sum of two sets) and (\ref{eq:dim}) can
be written as
\begin{equation}
||S_{r}||^{2}\le C||S_{s}||\bigg(\sqrt{\bigg|K+\left[-\frac{\log^{3}s}{s},\frac{\log^{3}s}{s}\right]\bigg|}+\frac{||S_{r}||r\log^{4}s}{s}\bigg).\label{eq:dim-2}
\end{equation}
Finally, note that (\ref{eq:dim-2}) does not actually require the
assumption $s/\log^{4}s>r$ because in the other case it holds trivially.
Hence (\ref{eq:dim-2}) holds for all $s>r$. Now we can put the tombstone.
\end{proof}

\begin{proof}
[Proof of theorem \ref{thm:doubly exponentially}] Let $K$ be the
support of the distribution $\sum c_{l}e(lx)$. We first claim that
we can assume without loss of generality that $S_{n_{k}}$ is bounded
on $K$. This uses the localisation principle exactly like we did
in the proof of proposition \ref{prop:measure}, but let us do it
in details nonetheless. Since $S_{n_{k}}(x)\to0$ everywhere $\sup_{k}|S_{n_{k}}(x)|$
is finite everywhere. Applying the Baire category theorem to the function
$\sup_{k}|S_{n_{k}}(x)|$ on $K$ we see that there is an open interval
$I$ such that $S_{n_{k}}$ is bounded on a dense subset of $K\cap I$,
and $K\cap I\ne\emptyset$. Continuity of $S_{n_{k}}$ shows that
they are in fact bounded on the whole of $K\cap I$. By the definition
of support of a distribution, we can find a smooth test function $\varphi$
supported on $I$ such that $\sum\widehat{\varphi}(l)c_{l}$ is not
zero. Let $d_{l}=c_{l}*\widehat{\varphi}$ (and hence $d$ is not
zero either). Then by the localisation principle (\ref{eq:Rajchman}),
$\sum_{-n_{k}}^{n_{k}}d_{l}e(lx)$ converges everywhere to zero and
is bounded on $K\cap I$, which contains the support of $\sum d_{l}e(lx)$.
Hence we can rename $d_{l}$ to $c_{l}$ and simply assume that $S_{n_{k}}$
is bounded on $K$.

We now construct a series $r_{i}$ as follows: we take $r_{1}=n_{1}$
and for each $i\ge1$ let $r_{i+1}$ be the first element of the series
$n_{k}$ which is larger than $r_{i}^{7/4}$. Because $n_{k+1}=n_{k}^{1+o(1)}$
we will have in fact that $r_{i+1}=r_{i}^{7/4+o(1)}$ and hence
\begin{equation}
r_{i}=\exp((7/4+o(1))^{i}).\label{eq:ridoubly}
\end{equation}
We now apply lemma \ref{lem:new riemann}  with $r_{\text{lemma \ref{lem:new riemann}}}=r_{i}$
and $s_{\text{lemma \ref{lem:new riemann}}}=r_{i+1}$. We get
\[
||S_{r_{i+1}}||\ge c||S_{r_{i}}||^{2}
\]
Denote this last constant by $\lambda$ for clarity (i.e.\ $||S_{r_{i+1}}||\ge\lambda||S_{r_{i}}||^{2}$).
Iterating the inequality $||S_{r_{i+1}}||\ge\lambda||S_{r_{i}}||^{2}$
starting from some $i_{0}$ such that $||S_{r_{i_{0}}}||>e/\lambda$
gives
\[
||S_{r_{i}}||\ge(\lambda||S_{r_{i_{0}}}||)^{2^{i-i_{0}}}>\exp(2^{i-i_{0}})
\]
Together with (\ref{eq:ridoubly}) we get 
\[
||S_{r_{i}}||\ge\exp((\log r_{i})^{1.2386+o(1)})
\]
(the number is $\approx\log2/\log\nicefrac{7}{4}$) which certainly
contradicts the boundedness of the $c_{l}$.
\end{proof}

\begin{proof}
[Proof of theorem \ref{thm:dim}]Denote $d=\dim_{\Mink}(K)$ (recall
that this is the upper Minkowski dimension). Assume by contradiction
that $c_{l}\not\equiv0$ and without loss of generality assume that
$c_{0}=1$ (if $c_{0}=0$, shift the sequence $c_{l}$ and note that
the condition $c_{l}\to0$ ensures that $S_{n_{k}}(x)\to0$ even for
the shifted sequence). 

Fix $s\in\mathbb{N}$ and let $\varphi$ be as in the proof of lemma
\ref{lem:new riemann}: let us remind the most important properties:
\begin{enumerate}
\item $\supp\varphi\cap K=\emptyset$;
\item $\supp(1-\varphi)\subset K+[-(\log^{3}s)/s,(\log^{3}s)/s]$; 
\item $\varphi(x)\in[0,1]$ for all $x\in[0,1]$; and
\item $|\widehat{\varphi}(l)|\le Cs\exp\Big(-c\sqrt{(|l|\log^{3}s)/s}\Big)$.
\end{enumerate}
From this we can get a lower bound for $||S_{s}||$. From $\supp\varphi\cap K=\emptyset$
we get
\[
\sum_{l=-\infty}^{\infty}c_{l}\widehat{\varphi}(l)=0
\]
so
\[
\int S_{s}\varphi=\sum_{l=-s}^{s}c_{l}\widehat{\varphi}(l)=-\sum_{|l|>s}c_{l}\widehat{\varphi}(l)
\]
giving
\[
\Big|\int S_{s}\varphi\Big|\le\sum_{|l|>s}Cs\exp\Big(-c\sqrt{(|l|\log^{3}s)/s}\Big)\le C\exp(-c\log^{3/2}s).
\]
By assumption $\int S_{s}=c_{0}=1$ so for $s$ sufficiently large
\[
\Big|\int S_{s}(1-\varphi)\Big|=1-O(\exp(-c\log^{3/2}s))>\nicefrac{1}{2}.
\]
Using Cauchy-Schwarz gives
\begin{align*}
\nicefrac{1}{2} & <||S_{s}||\sqrt{|\supp(1-\varphi)|}\\
 & \le||S_{s}||\sqrt{\left|K+\Big[-\frac{\log^{3}s}{s},\frac{\log^{3}s}{s}\Big]\right|}\le||S_{s}||\cdot\sqrt{s^{d-1+o(1)}}
\end{align*}
where in the last inequality we covered $K$ by intervals of size
$1/s$ \textemdash{} no more than $s^{d+o(1)}$ by the definition
of upper Minkowski dimension \textemdash{} and inflated each one by
$(\log^{3}s)/s$. We conclude that
\begin{equation}
||S_{s}||\ge s^{(1-d)/2+o(1)}\label{eq:lower bound}
\end{equation}
as $s\to\infty$.

In the other direction, fix some $r$ in our sequence and use (\ref{eq:dim-2})
to get:
\[
||S_{r}||^{2}\le C||S_{s}||\bigg(s^{(d-1)/2+o(1)}+\frac{||S_{r}||r\log^{4}s}{s}\bigg).
\]
Choose $s=(r||S_{r}||)^{2/(d+1)}$ (this makes the summands approximately
equal) and get
\begin{align}
\frac{||S_{s}||}{\sqrt{s}} & \ge||S_{r}||^{2}\cdot(r||S_{r}||)^{{\displaystyle \Big(-\frac{d}{d+1}+o(1)\Big)}}\nonumber \\
 & \stackrel{\textrm{\ensuremath{\mathclap{{(*)}}}}}{\ge}r^{{\displaystyle \Big(-\frac{d}{d+1}+\frac{1-d}{2}\cdot\frac{d+2}{d+1}+o(1)\Big)}}\label{eq:power of r}
\end{align}
where the inequality marked by $(*)$ follows from $||S_{r}||\ge r^{(1-d)/2+o(1)}$,
which is (\ref{eq:lower bound}) with $s$ replaced by $r$. When
$d<\frac{1}{2}(\sqrt{17}-3)$ the power of the $r$ in (\ref{eq:power of r})
is positive. This means that $||S_{s}||/\sqrt{s}\to\infty$, contradicting
the boundedness of the coefficients $c_{l}$.
\end{proof}

\section{\label{sec:Localisation}Localisation with bounded coefficients}

Our last remark is that there is a version of the localisation principle
suitable even when the coefficients of the series do not converge
to zero, but are still bounded. Let us state it first
\begin{thm}
\label{thm:local}Let $c_{l}$ be bounded and $n_{k}$ some sequence
and let $\varphi$ be a smooth function. Then there exists a subsequence
$m_{k}$ and two functions $a$ and $b$ such that 
\[
\varphi(x)\sum_{l=-m_{k}}^{m_{k}}c_{l}e(lx)-\sum_{l=-m_{k}}^{m_{k}}(c*\widehat{\varphi})(l)e(lx)+e^{im_{k}x}a(x)+e^{-im_{k}x}b(x)
\]
converges to zero uniformly.

Further, $a$ and $b$ have some smoothness that depends on $\varphi$
as follows:
\[
|\widehat{a}(l)|\le\sum_{|j|>|l|}|\widehat{\varphi}(j)|.
\]
and ditto for $b$.
\end{thm}

(recall that in the classic Rajchman formulation $a\equiv b\equiv0$
and $m_{k}$ can be taken to be $n_{k}$, one does not need to take
a subsequence).
\begin{proof}
Denote
\[
E_{n}(x)=\varphi(x)\sum_{l=-n}^{n}c_{l}e(lx)-\sum_{l=-n}^{n}(c*\widehat{\varphi})(l)e(lx).
\]
For $|j|>n$ only the first term appears in $\widehat{E_{n}}(j)$
and we get 
\[
\widehat{E_{n}}(j)=\sum_{l=-\infty}^{\infty}c_{j-l}\widehat{\varphi}(l)\mathbbm{1}\{|j-l|\le n\}
\]
and in particular $|\widehat{E_{n}}(n+r)|\le C\sum_{s\ge r}|\widehat{\varphi}(s)|$,
and similarly for $\widehat{E_{n}}(-n-r)$. For $|l|\le n$ the second
term also appears, but since it is simply the sum without the restriction
$|j-l|\le n$ the difference takes the following simple form:
\[
\widehat{E_{n}}(j)=-\sum_{l=-\infty}^{\infty}c_{j-l}\widehat{\varphi}(l)\mathbbm{1}\{|j-l|>n\}.
\]
Again we get $|\widehat{E_{n}}(n-r)|\le C\sum_{|s|\ge r}|\widehat{\varphi}(s)|$
and similarly for $\widehat{E_{n}}(-n+r)$. 

These uniform bounds for $|\widehat{E_{n_{k}}}(\pm n_{k}+r)|$ allow
us to use compactness to take a subsequence $m_{k}$ of $n_{k}$ such
that both $\widehat{E_{m_{k}}}(m_{k}+r)$ and $\widehat{E_{m_{k}}}(-m_{k}+r)$
converge for all $r$. Defining
\begin{align*}
a(x) & =-\sum_{r=-\infty}^{\infty}e(rx)\lim_{k\to\infty}\widehat{E_{m_{k}}}(m_{k}+r)\\
b(x) & =-\sum_{r=-\infty}^{\infty}e(rx)\lim_{k\to\infty}\widehat{E_{m_{k}}}(-m_{k}+r)
\end{align*}
the theorem is proved.
\end{proof}
Theorem \ref{thm:local} can be used to strengthen both theorems \ref{thm:doubly exponentially}
and \ref{thm:dim} to hold for bounded coefficients rather than for
coefficients tending to zero. But let us skip these applications and
show only how to use it to strengthen proposition \ref{prop:measure}.
\begin{thm}
Let $\mu$ be a measure and let $n_{k}$ be a series such that
\[
\lim_{k\to\infty}S_{n_{k}}(\mu;x)=0\qquad\forall x.
\]
Then $\mu=0$.
\end{thm}

\begin{proof}
Let $K$ be the support of $\mu$ and let, as in the proof of proposition
\ref{prop:measure}, $I$ be an interval such that $S_{n_{k}}(\mu)$
is bounded on $I$ and $I\cap K\ne\emptyset$. Let $\varphi$ be a
smooth function supported on all of $I$. We use theorem \ref{thm:local}
to find a subsequence $m_{k}$ of $n_{k}$ and an $a$ and a $b$
such that
\begin{equation}
\varphi S_{m_{k}}(\mu)-S_{m_{k}}(\varphi\mu)+e^{im_{k}x}a+e^{-im_{k}x}b\to0.\label{eq:a and b}
\end{equation}
This has two applications. First we conclude that $\varphi\mu\not\in L^{2}$.
Indeed, if we had that $\varphi\mu\in L^{2}$ then we would get that
$\varphi S_{m_{k}}(\mu)\to0$ pointwise while $S_{m_{k}}(\varphi\mu)\to\varphi\mu$
in measure, which can only hold if $\varphi\mu\equiv0$ (also $a$
and $b$ need to be zero, but we do not need this fact). This contradicts
our assumption that $I\cap K\ne\emptyset$ and that $\varphi$ is
supported on all of $I$.

Our second conclusion from (\ref{eq:a and b}) is that $S_{m_{k}}(\varphi\mu)$
is bounded on $I\cap K$, which is the support of $\varphi\mu$. From
here the proof continues as in the proof of proposition \ref{prop:measure}.
\end{proof}

\subsection*{Acknowledgements}

Both authors were supported by their respective Israel Science Foundation
grants. GK was supported by the Jesselson Foundation, and by Paul
and Tina Gardner.

\end{document}